\documentclass[11pt,onecolumn]{article}
\usepackage{amscd}
\usepackage{times}
\usepackage{amsmath}
\usepackage{amssymb}
\usepackage{xspace}
\usepackage{theorem}
\usepackage{graphicx}
\usepackage{ifpdf}
\usepackage{url,hyperref}
\usepackage{latexsym}
\usepackage{euscript}
\usepackage{xspace}
\usepackage[all]{xy}
\usepackage{color}
\usepackage{makeidx}
\usepackage{tikz}
\usepackage{enumitem}
\long\def\remove#1{}
\newtheorem{theorem}{Theorem}[section] 

\newtheorem{obs}[theorem]{Observation}

\newtheorem{definition}[theorem]{Definition}
\newtheorem{proposition}[theorem]{Proposition}
\newenvironment{proof}{{\em Proof:}}{\hfill{\hfill\rule{2mm}{2mm}}}

\newcommand {\mm}[1] {\ifmmode{#1}\else{\mbox{\(#1\)}}\fi}

\newcommand{\img}{\mathrm img}
\newcommand{\supp}{\mathrm supp}
\newcommand{\coker} {\mathrm coker}

\newcommand{\rank}                {\mm {\rm rank}}

\newcommand{\cancel}[1]

\begin{document}

\title {"Barcodes" for continuous maps and a brief introduction to Alternative  Morse Theory}

\author{
Dan Burghelea  \thanks{
Department of Mathematics,
The Ohio State University, Columbus, OH 43210,USA.
Email: {\tt burghele@math.ohio-state.edu}}
}
\date{}

\maketitle

\hskip 1in {\it Dedicated to Valentin Poenaru  
for his 90-th anniversary}

\begin{abstract}

This paper reviews the description of "bar codes" for a continuous real-valued map  $f :X\to \mathbb R$ and explains how to recover the Morse complex of a Morse function from them. In this presentation the bar codes appear as the support of two  vector-space valued maps, one defined on the Euclidean plane $\mathbb R^2$ and the other on the "above diagonal" half plane $\mathbb R^2_+.$  
\vskip .1in
\end{abstract} 

\setcounter{tocdepth}{1}

\section {Introduction}

Classical Morse Theory (C.M.T) considers smooth functions $f: M\to \mathbb R$ on  smooth manifolds $M$ whose all critical points are non-degenerate (generic smooth functions)  and, under some conditions, relates them  to the homology of the underlying manifold.   This is done by providing a family of chain complexes, each associated to some additional data, but all isomorphic \footnote {when the chain complex is a complex of $\kappa-$vector spaces,  as it will be the case in this paper,  the unicity will  follow  from calculations; when the chain complex is of modules over an arbitrary ring this was established in \cite {CR} } and therefore called {\it Morse complex} which calculate the homology of the manifold.
The Morse complex explains the relations between the number $c_r$ of critical points of the  index $r$ and the dimension $\beta_r$ of the $r-$homology vector spaces  (the $r-$Betti numbers) and detects existence of instantons (e.g. isolated trajectories between rest points) for a vector field which admits the function $f$ as Lyapunov function. The Morse complex is also equipped with an $\mathbb R-$ filtration by sub complexes which calculate the homology of the piece $f^{-1} ((-\infty, t]).$ This filtration is locally constant in $t$ outside the critical values  of $f.$ 
The elementary Morse theory, as summarily reviewed in Section 3, has  extensions to various type  of smooth infinite dimensional manifolds and smooth Whitney stratified spaces and  to the case when the set  of non degenerate critical points consists of submanifolds rather than points.    
\vskip .2in

Alternative Morse Theory (A.M.T) begins with a continuous map $f: X\to \mathbb R$  and refines the set $CR(f)$ of critical values of $f$ \footnote {the values $t$ for which the homology of the level $f^{-1} (t)$ changes}, into a collection of four types of intervals which, in topological persistence homology theory \footnote {proposed in \cite {ELZ} further extended in \cite {CSM} and \cite {BD11}} are referred to as {\it barcodes}.  
      
In our work they  appear as the points in the support of two types of vector spaces-valued maps $\hat \delta^f_r$ and $\hat \gamma^f_r$ defined on $\mathbb R^2$ and $\mathbb R^2\setminus \Delta,$ 
where $\Delta$  denotes the diagonal in $\mathbb R^2,$ associated to  each $r\in \mathbb Z_{\geq 0}$ and field $\kappa.$ The multiplicity of each bar code is the dimension of the corresponding vector space, possibly infinite.
Since in this paper only the restriction of the map $\hat\gamma^f_r$ to the $\mathbb R^2_+:=\{x,y\in \mathbb R^2 \mid x<y\}$ will appear,  one denotes  this restriction by ${^+ \hat\gamma}^f_r.$
The restriction to the below diagonal half plane, $\mathbb R^2_-:= \{x,y\in \mathbb R^2 \mid x<y\},$  denoted in \cite{Bu3} by ${^-\hat \gamma}^f_r,$ can be derived from ${^+ \hat\gamma}_r$ by the formula 
${^- \hat\gamma}^f_r(a,b)= {^+ \hat\gamma}^{-f}_r (-b,-a).$

These maps (actually the points of their supports with their multiplicity) permit to define a chain complex of $\kappa-$vector spaces.  
This chain complex is also equipped  with an $\mathbb R-$filtration determined  by the critical values and when considered for  a Morse  function  is isomorphic to the Morse complex with its filtration. 
In particular, as in the case of a Morse function $f$, from this complex, determined by bar codes, one can recognize  the number of critical points corresponding to each critical value,  the homology of the manifold,  of the sub-levels $f^{-1}((-\infty, t]),$ of the levels $f^{-1} (t)$ etc. and one detects presence of instantons;  all these    
in a considerably larger class of situations than in the C.M.T.  

 A key merit of the theory is that the numerical invariants which appear are {\it computer friendly}, i.e. when considered for nice spaces (for example finite simplicial complexes and simplicial maps)  they can be calculated by implementable algorithms. cf \cite{Bu}.

The purpose of this paper is to provide the definition of the maps $\hat \delta^f_r,$  ${^+ \hat \gamma}^f_r,$   ${^+ \hat \mu}^f_r,$  ${^+ \hat \omega}^f_r$and   ${^+ \hat \lambda}^f_r$ for an arbitrary continuous map, relate them to the homology of the underlying space when $f$ is tame  ,
and present the arguments to establish the isomorphism of the chain complex defined  by $\hat \delta^f_r,$ ${^+ \hat \gamma}^f_r$ and   ${^+ \hat \mu}^f_r$ with the Morse complex when associated to a Morse function, however in this paper we do this only in the case that the involved chain complexes are of finite dimension vector spaces, in particular for the subcomplexes describing the $\mathbb R-$filtration.
 The concluded isomorphism is not canonical. The existence of an isomorphism compatible with the filtration  is algebraically  more subtle and will be discussed in subsequent work.

The paper  begins with a few observation about chain complexes in Section 2,  and a brief review of elementary Morse theory in Section 3.
Section 4 provides the definitions of the maps $\hat \delta_r^f,$ ${^+ \hat \gamma}^f_r,$ ${^+ \hat \mu}^f_r,$  ${^+ \hat \omega}^f_r$and   ${^+ \hat \lambda}^f_r$ and of the associated chain complex $(C^{\delta, \gamma, \mu}_\ast,  \partial^{\delta, \gamma, \mu}_\ast)$ and of its $\mathbb R-$filtration  
$(C^{\delta, \gamma,\mu}(t)_\ast. \partial^{\delta, \gamma, \mu}_\ast)$ and  formulates the precise relation between the homology and the vector space-valued maps, cf Theorem \ref{T44}.

Section 5 collects results needed to establish the non canonical isomorphism  of $(C^{\delta, \gamma, \mu}(t)_\ast. \partial^{\delta, \gamma, \mu}_\ast )$
 with the $t-$subcomplex of the Morse complex. 
The theory can be extended to A.M-N.T (M-N.T stands for Morse-Novikov theory cf  \cite {Bu2}).

{\bf Note}
In 1961-62, as a second year undergraduate student I met Valentin Poenaru at that time a young and charismatic researcher at the Mathematical Institute of the Romanian Academy in Bucharest. In view of my interest in topology he has invited me  to attend his seminar in differential topology; at that time the seminar discussed Morse Theory. Despite our rather short intersection in Bucharest (he soon left Romania) directly, through  his seminar,  or may be indirectly, by  his reputation, he has much encouraged  my wish  to become a topologist.
I am pleased to dedicate to him this paper, containing a few considerations about, or at least related to Morse theory, a topic in topology, which has remained directly or indirectly in the back of much of my mathematics and turns out to have more and more relevance outside topology. 

\vskip  .2in

\section {Chain complexes of $\kappa-$vector spaces}
For a chain complex  of $\kappa-$vector spaces 
$$(C_\ast, \partial _\ast)\equiv \{ \xymatrix{\cdots\ar[r]& C_{n+1}\ar[r]^{\partial _{n+1}} & C_n\ar[r]^{\partial _n}& C_{n-1} \ar[r] &\cdots}, \  \ \partial _{k-1}\cdot \partial _k=0\}$$   
a Hodge decomposition is provided by a decomposition of each $C_n$ 
$$C_n= C_n^- \oplus H_n\oplus C_n^+$$ s.t. $\partial_n: C_n\to C_{n-1}$ is described  as the matrix  
$$\partial _n= \begin{bmatrix}
0&0&0\\
0&0&0\\
\underline \partial_n&0&0
\end{bmatrix}
$$
with $\underline \partial_n$ an isomorphism; clearly $H_n( C_\ast, \partial_\ast) = H_n.$
Such decomposition exists and  is unique up to a non canonical isomorphism. When $C_n$ are finite dimensional  one writes $c_n= \dim C_n,$   $\beta_n= \dim H_n$  and $\rho_n= \rank \ \partial _n$
and these numbers are related by $$c_n= \beta_n+ \rho_n +\rho_{n-1}.$$

The existence of Hodge decompositions and the equality above implies : 
\begin{obs}\label {O21}\ 
Two of these three sets of numerical invariants determine up to  a non-canonical isomorphism a chain complex of finite dimensional vector spaces. 
\end{obs}

The proof is straightforward. For more details, if necessary, see \cite{Bu2} section 8.

\vskip .1in
An $\mathbb R-$filtration of $(C_\ast, \partial _\ast)$ consists of a family of subcomplexes $(C_\ast(t), \partial _\ast) \subset (C_\ast, \partial _\ast),$ i.e.$\partial _\ast (C_\ast(t) \subset C_{\ast-1}(t),$  indexed by $t\in \mathbb R$ s.t. 
$(C_\ast(t), \partial _\ast) \subseteq (C_\ast(t'), \partial _\ast )$ for  $t<t' $ and $\bigcup_t (C_\ast(t), \partial _\ast)= (C_\ast, \partial _\ast).$  The complex $(C_\ast(t), \partial _\ast)$ is referred to as the $t-$filtration
or the $t-$subcomplex of $(C_\ast, \partial _\ast).$

\section{ Classical Morse theory}

Morse theory for finite dimensional smooth manifolds considers proper smooth maps  bounded from below with all critical points non degenerate \footnote { a critical points $x\in  Cr(f):=\{ x\in M \mid  df(x)=0\}$ is non degenerate if 
the Hessian of $f$ at $x,$ i.e. $\partial ^2 f/ \partial x_i, \partial x_j (x),$ in a coordinate system (ant then in any) is a non degenerate quadratic form; the index of $x$ is ihe number of negative eigenvalues of the Hessian}, 
 
 To such function and to an arbitrary field $\kappa$ one associates a collection of  chain complexes of finite dimensional $\kappa-$ vector spaces equipped with an $\mathbb R-$filtration, locally constant for $t\in \mathbb R\setminus CR(f),$  which calculates the homology of $M$ and of $f^{-1} ((-\infty, t])$  
 whose  components $C_n$ resp. $C_n(t)$ are the $\kappa-$vector spaces generated by the critical points of index $n$ resp. the critical points of index $n$ of  critical values smaller or equal to $t$. 
 The boundary maps $\partial_n: C_n\to C_{n-1}$ depends on additional data. but different additional data provide isomorphic complexes. 
 Any such complex will be named {\it Morse complex of $f$}.
 \vskip .1in
 
 The additional data  contains  a smooth vector field  $X$ which has $f$ as good Lyapunov function,  in Milnor terminology a {\it gradient like vector field}, cf. \cite {Mi};  precisely,  
  $X(f) (x) <0$  iff $x\in M\setminus  Cr(f)$ with $Cr(f)= \{x\in M \mid df(x)=0\}$ and for any $x\in Cr(f)$ one can find a chart $\varphi_x: (U_x, x)\to (\mathbb R^n, 0),$ $U_x$ open neighborhood  of $x$ in $M,$ s.t. 
  
  \begin{equation*}
 \begin{aligned}
  f\cdot \varphi_x^{-1} (x_1, \cdots x_n)=& -1/2 \sum _{i\leq k} x^2_i + 1/2 \sum _{i \geq k+1} x^2_i\\  
  \varphi _x ^\ast X= &  \sum _{i\leq k} x_i \partial _{x_i}- \sum _{i\geq k+1} x_i \partial _{x_i}.
  \end{aligned}
  \end{equation*}
  Equivalently,  $X= -grad_g f$  for $g$  a complete Riemannian metric on $M$ which in the neighborhood
  of each critical point is flat.
\vskip .1in

For a smooth vector field $X$ on a smooth manifold  $M^n$  call  {\it rest point} a point $x\in M$  s.t. $X(x)=0$ and {\it trajectory}  a smooth map $\gamma : \mathbb R\supset U\to M$  which satisfies $d \gamma(t)/ d t= X(\gamma(t)).$   

Denote by $\gamma_y$ the maximal trajectory with $\gamma_y(0)=y$, $y\in M,$  and define  
$$W_x^\mp:= \{ y\in M\mid \lim_{t\to\mp \infty} \gamma_y(t)= x\},$$ the unstable (-) resp. stable (+) set of the rest point $x.$ If the vector field is as above then these sets are actually submanifolds diffeomorphic to the euclidean space of dimension  $ index (x)$ resp. $ n- index (x).$

An  {\it additional data } consists of a gradient like vector field  $X$ which has the unstable manifolds transversal  to the stable manifolds and has $f$ as a good Lyapunov function plus a collection  $\mathcal O= \{o_x, x\in Cr(f)\}$ of orientations $o_x$ for any unstable manifold $W^-_x.$ 
 Gradient like vector fields exist, and any vector field with hyperbolic rest points  having $f$ as Lyapunov function  can be perturbed arbitrary little on arbitrary small neighborhood of the rest points,  but still remaining the same in a smaller neighborhood of these points, to get the stable and unstable manifolds transversal cf. \cite {Sm} or \cite {BFK}.

If $x$ is a critical point of index $k$ and $y$ a critical point of index $(k-1)$ then $W^-_x\cap W^+_y$ is a union of components $\gamma,$ each submanifold of dimension one called {\it instanton} (= isolated trajectory between rest points).  

The orientation $o_x$ compared to the orientation $ o_y$  followed by the orientation from  $x$ to $y$   along the instanton $\gamma,$ provides a sign, $\epsilon (\gamma)= \pm 1,$ and then one defines the {\it algebraic cardinality} $I(x,y)= \sum _\gamma \epsilon (\gamma)$ of the set of instantons from $x$ to $y$.

The additional data $(X, \mathcal O)$ provides the vector spaces $C^{X, \mathcal O}_k$ generated by rest points of $X,$ the same as  the critical points of $f$ of index $k,$ hence equipped with a base.
The  subspace $C^{f, X, \mathcal O}_k(t)$ is generated by the rest  points $x$ with $f(x)\leq t$ and the linear map  $\partial^{X, \mathcal O}_k: C^{X, \mathcal O}_k\to C^{X, \mathcal O}_{k-1}$  is given by the matrix with entries $I(x,y).$ 

The main theorem of elementary Morse theory claims that if $f$ is a proper smooth function bounded from below with all critical points non degenerate and $(X,\mathcal O)$ is an additional data, then  
these vector spaces and linear maps described above define a chain complex $(C^{X,\mathcal O}_\ast,\partial ^{X,\mathcal O}_\ast)$ with $\mathbb R-$filtration $(C^{X,\mathcal O}_\ast(t), \partial^{X, \mathcal O} _\ast)$ 
which calculates the homology  of $X$ and of $f^{-1} ((-\infty,t]).$ The $t-$filtration  subcomplex has all vector spaces of finite dimension. 
This implies the famous Morse inequalities, the independence of additional data up to a non canonical isomorphism, i.e. the complexes derived from different additional data are isomorphic, cf. \cite {CR}, 
 as well as the fact that $\rank\ \partial _k\ne 0$ implies existence of instantons for any vector field having $f$ as Lyapunov function. 

We refer to any of these complexes  $(C_\ast ^{ X, \mathcal O}, \partial ^{X, \mathcal O}_\ast)$ and the sub complexes  $(C_\ast ^{f, X, \mathcal O}(t), \partial ^{X, \mathcal O}_\ast)$ as Morse complexes. The A.M.T as proposed has an extension to A.M-N.T (Alternative Morse-Novikov theory)  but this is not discussed in this paper. 

\section {The maps $\hat \delta^f_r,$  $^+\hat \gamma^f_r$ and $^+\mu^f_r,  ^+\lambda^f_r, ^+\omega^f_r $ and the associated chain complex  $(C^{\delta,\gamma,\mu}_\ast, \partial ^{\delta, \gamma, \mu}_\ast)$ } 

{\bf The maps $\hat \delta^f_r, ^+\hat \mu^f_r, ^+\hat \gamma^f_r,  ^+\hat \lambda^f_r, ^+\hat \omega^f_r $} 
\vskip .1in
Let   $f: X\to \mathbb R$, be a  continuous real-valued map and let $H_\ast$  denote the singular homology \footnote {or any other homology theory which satisfies the Eilenberg-Steenrod axioms and commutes with the direct limits} 
 with coefficients in a fixed field $\kappa.$

 \noindent For any $a\in \mathbb R$ denote by
 
 \begin{equation}
 \begin{aligned}
X^f_a:= f^{-1} ((-\infty,a]), & \ X_f^a:= f^{-1} ([a, \infty))\\
X^f_{<a}:= f^{-1} ((-\infty,a)),& \ X_f^{>a}:= f^{-1} ((a, \infty)).
\end{aligned}
 \end{equation}  
The real number  $a\in \mathbb R$  is called  {\it regular value } 
if  $$H_\ast(X^f_a, X^f_{<a}) \oplus H_\ast(X_f^a, X_f^{>a})
= 0$$ and {\it critical value} if not regular. 
Denote the set of critical values by $CR(f).$

For any $a\in \mathbb R$ denote by 
\begin{equation} \label{E2}
 \begin{aligned}\
\mathbb I^f_a(r):= \img (H_r(X^f_a) \to H_r(X)), &\  \mathbb I_f^a(r): =\img (H_r(X_f^a) \to H_r(X))\\
\mathbb I^f_{<a}(r):= \img (H_r(X^f_{<a}) \to H_r(X)), &\  \mathbb I_f^{>a}(r): =\img (H_r(X_f^{>a}) \to H_r(X))
\end{aligned}
 \end{equation}  
with the arrows representing the inclusion induced linear maps. 

Note that $H_r(X_{<a})= \varinjlim_{\epsilon \to 0}H_r(X_{a-\epsilon})$
and then  $\mathbb I_{<a} (r)= \varinjlim_{\epsilon \to 0}\mathbb I_{a-\epsilon}(r).$ 
Similarly $H_r(X^{>a})= \varinjlim_{\epsilon \to 0} H_r(X^{a+\epsilon})$ and then  $\mathbb I^{>a}(r)= \varinjlim_{\epsilon \to 0}\mathbb I^{a + \epsilon}(r).$
\vskip .2in
 
For any $a,b \in  {\mathbb R}$ let 
\begin{equation}
\begin{aligned}
\mathbb F^f_r(a,b):= &\mathbb I^f_a(r)\cap \mathbb I^b_f(r), \\
 \mathbb F^f_r (<a,b):= &\mathbb I^f_{<a}(r)\cap \mathbb I^b_f(r),\\
  \mathbb F^f_r(a,>b):= &\mathbb I^f_a(r)\cap \mathbb I^{>b}_f(r),
\end{aligned}
\end{equation}
and define: 
\begin{equation}\label {E4}                     
\boxed{\hat\delta^f_r(a,b):= \mathbb F_r(a,b) / (\mathbb F_r(<a,b) + \mathbb F_r(a, >b))}.\end{equation}                                                                                                                                                                                                                 
 
 Let $$\delta^f_r(a, b):= \dim \hat\delta^f_r(a,b) \in   \mathbb Z_{\geq 0}\cup \infty .$$
\vskip .2in
We also introduce 
\begin{equation}\label {E5}
\begin{aligned}
\mathbb I^f_\infty(r):= &\cup _{a\in \mathbb R} \mathbb I^f_a(r)= H_r(X),\\
\mathbb I_f^{-\infty}(r):=& \cup _{a\in \mathbb R} \mathbb I_f^a(r)= H_r(X),\\
\mathbb I^f_{-\infty}(r):= &\cap _{a\in \mathbb R} \mathbb I^f_a(r),\\
\mathbb I_f^{\infty}(r):=& \cap _{a\in \mathbb R} \mathbb I^a_f(r)
\end{aligned}
\end{equation}
 and then define 
  \begin{equation}\label {E6}
\boxed{{^+\hat\mu}^f_r(a):= (\mathbb I^f_a(r)\cap \mathbb I_f^{\infty}(r)) / (\mathbb I^f_{<a}(r)\cap \mathbb I_f^{\infty}(r))}
\end{equation}
and $$^+\mu^f_r(a):=\dim
^+\hat\mu^f_r(a) \in   \mathbb Z_{\geq 0}\cup \infty .$$ 

\begin{obs}\

\noindent The assignment $\hat \delta_r$ defines the map 

 $\hat \delta: \mathbb R^2\rightsquigarrow \kappa-\rm{Vector\ spaces}$

\noindent and the assignment ${^+\hat\mu}_r$ defines the map 

${^+\hat \mu}_r: \mathbb R\rightsquigarrow \kappa-\rm{Vector\ spaces},$

\noindent which in view of the definition of regular / critical values have the supports contained in $CR(f)\times CR(f)$ and. $CR(f)$ resp..
If $f$ is bounded from above, hence $\varprojlim_{x\to \infty} \mathbb I_f^{x}(r)=0,$ then ${^+\hat\mu}^f_r (a)= 0$ for any $a\in \mathbb R.$ 
\end{obs}

\vskip .3in

For $a, b \in \mathbb R$ let 
\begin{equation}\label {E7}
\begin{aligned}
\mathbb T^f_r(a,b):= &\ker ( H_r(X^f_a)\to H_r(X^f_b)) \  \ \rm {if} \ &a\leq b \\
 \mathbb T^f_r(<a,b):= &\ker ( H_r(X^f_{<a})\to H_r(X^f_b))\  \ \rm {if} \ &a\leq b\\
  \mathbb T^f_r(a,<b):= &\ker ( H_r(X^f_a)\to H_r(X^f_{<b}))\  \ \rm {if} \ &a<b 
  \end{aligned}
\end{equation}
and define:
\begin{equation}\label {E8}
\boxed{^+ \hat\gamma^f_r(a,b):= \mathbb T_r(a,b) / ( \iota \mathbb T_r(<a,b) + \mathbb T_r(a, <b))}, \end{equation}
with $$^+\gamma^f_r(a, b):= \dim \hat\gamma^f_r(a,b) \in   \mathbb Z_{\geq 0}\cup \infty $$
where $ \iota  : \mathbb T_r(<a,b)\to \mathbb T_r(a,b)$  is the inclusion induced linear map. 
When we want to specify the source and the target we write $^c \iota^b_a: \mathbb T_r(a,c)\to \mathbb T_r(b,c), a <b<c $ for the inclusion induced linear map $\iota.$ The map $\iota$ in the definition of $^+\hat \gamma^f_r$ above is actually $^b\iota_{<a}^a.$

For $y <a$ introduce  
 \begin{equation}\label {E9}
\boxed{^+\hat \lambda^f_r(a) := \cap _{y<a}\img ( ^a \iota_{y}^{<a} : \mathbb T_r(y, a)\to \mathbb T_r(<a, a)).}
\end{equation}
 and $$^+\lambda^f_r(a):= \dim ^+ \hat \lambda^f_r(a) \in   \mathbb Z_{\geq 0}\cup \infty .$$
 
  Note that $^a \iota_{y}^{<a} (\mathbb T_r(y, <a))=0$  \footnote {note that $\mathbb T_r(y,<a)\subseteq \mathbb T_r (y,a)$}, and  therefore 
 $$\img ( ^a \iota_{y}^{<a}:\mathbb T_r(y, a)/\mathbb T_r(y, <a)\to \mathbb T_r(<a, a))= \img ( ^a \iota_{y}^{<a} : \mathbb T_r(y, a)\to \mathbb T_r(<a, a)).$$
 \vskip.1in
 
 For $x>a$ observe that 
 $ \mathbb T_r((a,x) / ^x \iota^a_{<a} \mathbb T_r(<a,x) \subseteq \coker (H_r(X_a)\to H_r(X_a))$
 and then introduce 

 \begin{equation}\label {E10}
\boxed{^+\hat \omega^f_r(a) := \cap _{x>a} ( \mathbb T_r(a,x) / ^x \iota^a_{<a} \mathbb T_r(<a,x))}
\end{equation}
 and  $$^+\omega^f_r(a):= \dim ^+ \hat \omega^f_r(a) \in   \mathbb Z_{\geq 0}\cup \infty .$$

\begin{obs}\
\noindent The assignment $^+\hat \gamma^f_r$ defines the map 

$^+\hat \gamma^f_r: \mathbb R^2_+\rightsquigarrow \kappa-\rm { Vector\ spaces}$

\noindent and the assignments ${^+\hat \lambda}^f_r$  and ${^+\hat \omega}^f_r$ define the maps 

${^+\hat \lambda}^f_r: \mathbb R\rightsquigarrow \kappa-\rm { Vector\ spaces}$ and 

${^+\hat \omega}^f_r: \mathbb R\rightsquigarrow \kappa-\rm { Vector\ spaces}$ 

\noindent which in view of the definition of regular / critical values have the supports contained in $CR(f)\times CR(f)$ and $CR(f).$ 

If $f$ is bounded from below then ${^+\hat\lambda}^f_r (a)=0.$ If $f$ is homologically tame, cf Definition 4.4 below, then ${^+\hat \omega}^f_r (a)=0.$
\end{obs}

\vskip .2in 
{\bf Relationship with barcodes} 
\vskip .1in
The point $(a,b)\in \supp \delta ^f_r$ corresponds to what in \cite {Bu} is called an $r-$closed bar code $[a,b]$ of multiplicity $\dim \hat \delta^f_r(a,b)$ when $a\leq b$ and to an $(r-1)-$open bar code $(b,a)$ when $a>b$ 
\footnote {this because $a$ and $b$ should denote the ends of an interval} with the multiplicity $\dim \hat \delta_r(a,b)$.  Similarly the point $(a,b)$ with $a<b$ in $\supp ^+\gamma^f_r$ corresponds to what in \cite {Bu} is called an $r-$closed-open  bar code $[a,b)$ of  multiplicity $\dim {^+ \hat \gamma}^f_r(a,b).$

\vskip.2in
{\bf The chain complexes  $(C^{\delta,\gamma,\mu}_\ast, \partial ^{\delta, \gamma, \mu}_\ast)$ and $(C^{\delta,\gamma. \mu}_\ast(t) , \partial ^{\delta, \gamma, \mu}_\ast)$}
\vskip .1in
Define 
\begin{equation} \label {E11}
\begin{aligned}
C^-_n:=&\bigoplus _{\{(a,b)\in \mathbb R^2
\mid a<b\}} \ ^+ \hat \gamma^f_{n-1} (a,b)\\
H_n:=&\ \ \ \bigoplus _{\{(a,b)\in \mathbb R^2
\}} \  \  \ \hat \delta^f_n (a,b) \oplus \bigoplus _{\{a\in \mathbb R\}} \ ^+\hat \mu_n(a)\\
C^+_n:=&\ \bigoplus _{\{(a,b)\in \mathbb R^2
\mid a<b\}} \ ^+ \hat \gamma^f_{n} (a,b)\\
C_n:= &C_n^-\oplus H_n\oplus C^+_n
\end{aligned}
\end{equation}
and $\partial_n: C_n\to C_{n-1}$ by the matrix  
$$\partial _n= \begin{bmatrix}
0&0&0\\
0&0&0\\
id&0&0
\end{bmatrix}
\quad
$$

with the $t-$filtration provided by 
\begin{equation} \label{E12}
\begin{aligned} C_n(t):=
&\bigoplus _{\{ (a,b)\in CR(f)\times CR(f)\mid a<b \leq t \} } \  ^+{\hat{\gamma} }_{n-1} (a,b) \oplus \\
&\bigoplus_{\{(a,b)\in CR(f)\times CR(f) \mid  a \leq t\}} \hat \delta^f_n (a,b)
\oplus \bigoplus _{\{(a,b)\in CR(f)\times CR(f) \mid a \leq t <b\}}\  ^+ \hat \gamma^f_n(a,b)\ \  ) \oplus \bigoplus _{\{a\in CR(f)\mid a\leq t\}} \ ^+ \hat \mu _r(a)\oplus \\
 &\bigoplus _{\{ (a,b)\in CR(f)\times CR(f)
 \mid a<b \leq t\}}\  ^+\hat \gamma^f_n(a,b) .
\end{aligned}
\end{equation}
Clearly $\partial _n( C_n(t)) \subset C_{n-1} (t).$ 
 This  sub complex has the Hodge decomposition with the components  

\begin{equation}\label{E13}
\begin{aligned}
 C_n^- (t)= &\bigoplus _{\{ (a,b)\in CR(f)\times CR(f)
\mid a<b \leq t \} }  \ ^+{\hat \gamma}_{n-1} (a,b) \\
H_n(t)= &
 \bigoplus_{\{(a,b)\in CR(f)\times CR(f)
 \mid  a \leq t\}} \hat \delta^f_n (a,b)\oplus \bigoplus _{\{a\in \mathbb R\mid a\leq t\}} \ ^+ \hat \mu _r(a)
\oplus \bigoplus _{\{(a,b)\in CR(f)\times CR(f) \mid a \leq t <b\}}\  ^+ \hat \gamma_n(a,b)\\ 
C_n^+ (t)=&\bigoplus _{\{ (a,b)\in CR(f)\times CR(f) \mid a<b \leq t\}}\  ^+\hat \gamma^f_n(a,b).
\end{aligned}
\end{equation}

From now on one denotes  the above  complex (\ref {E11})  by $(C_\ast ^{\delta, \gamma,\mu}, \partial ^{\delta, \gamma,\mu}_\ast)$ and the $t-$filtration sub complex (\ref {E13}) by $(C^{\delta, \gamma, \mu}_\ast (t), \partial ^{\delta, \gamma, \mu}_\ast).$

\vskip .2in 

{ \bf  The case $H_r(X_a, X_{<a})$ is  finite dimensional}
\vskip .1in
Consider the obvious linear maps  

\begin{enumerate}
\item 
$\xymatrix{\mathbb F_r (a,b)/ \mathbb F_r( <a,b)\ar@{>->}[r] & \mathbb I_a(r)/ \mathbb I_{<a}(r) & \coker( H_r(X_{<a})\to H_r(X_a)) \ar@{->>}[l] \ar@{>->}[r] &H_r(X_a, X_{<a}) }$
\newline with the last arrow induced from the long exact sequence of the pair $(X_a, X_{<a}),$  
\item
$\xymatrix{\mathbb T_r (a,b)/\iota \mathbb T_r(<a, b)\ar@{>->}[r] & \coker( H_r(X_{<a})\to H_r(X_a)) \ar@{>->}[r] &H_r(X_a, X_{<a}) } $
with the last arrow induced from the long exact sequence of the pair $(X_a, X_{<a})$ and the first induced  from the inclusions 
$\mathbb T_r(a,b)\subseteq H_r(X_a)$  and $\mathbb T_r(<a,b)\subseteq H_r(X_{<a}),$
\item
$\xymatrix{\mathbb T_{r-1} (a,b)/\mathbb T_{r-1}(a, <b) & \coker( H_r(X_{<b}, X_a)\to H_r(X_b, X_a)) \ar@{->>}[l]\ar@{>->}[r] &H_r(X_b, X_{<b}) } $
\end{enumerate}
with the last arrow induced from the long exact sequence of the triple $(X_a\subseteq X_{<b} \subset X_b)$  and the first derived from 
the diagram 

$$\xymatrix { 
H_r(X_{<b},X_a)\ar@{->>}[d] \ar[r]&H_r(X_b,X_a) \ar@{->>}[r]\ar@{->>}[d]&\coker ( H_r(X_{<b},X_a) \to H_r(X_b,X_a))\ar@{->>}[d] \\
\mathbb T_{r-1}(a,<b)\ar[r]& \mathbb T_{r-1}(a,b)\ar@{->>}[r] &\mathbb T_{r-1}(a,b) / \mathbb T_{r-1}(a,<b)
}$$
The arrow
 $\xymatrix {\ar @{>->}[r]&}$ resp. $\xymatrix {\ar @{->>}[r]&}$ indicates  injective  resp. surjective linear map. 
\vskip .1in
As an immediate consequence one has 

\begin {obs}\label {O43}\

\begin{enumerate} [label=(\roman*)]
\item  $\dim H_r( X^f_a, X^f_{<a}) \geq \dim \mathbb F_r(a,b)/ \mathbb F_r( <a,b)\geq \dim \mathbb F_r (a\times [b, b')\geq \delta^f_r(a,b)$  
is a consequence of 1. above.
\item $\dim H_r( X^f_a, X^f_{<a}) \geq \dim  \coker( H_r(X_{<a})\to H_r(X_a)) \geq \dim \mathbb T_r(a,x)/ \iota \mathbb T_r( <a,x)\geq ^+\omega^f_r(a)$ 
is a consequences of 2. above.
\item  $\dim H_r( X^f_a, X^f_{<a}) \geq \dim \mathbb T^f_r(a,b) / \iota \mathbb T_r( <a,b) \geq \dim \mathbb T_r(a\times (b', b])\geq \dim ^+\hat \gamma_r(a,b)$ 
is a consequence of 2. above.
\item  $\dim H_r( X^f_b, X^f_{<b}) \geq \dim \mathbb T^f_{r-1}(a,b)/ \mathbb T_{r-1}( a,<b) \geq \dim \mathbb T_{r-1}((a', a]\times b) \geq \dim ^+\hat \gamma_{r-1}(a,b)$ 
is a consequence of 3. above.
\item $\dim H_r(X_a, X_{<a}) \geq \dim \coker (H_r(X_{<a} \to H_r(X_a))\geq \dim \mathbb I_a(r)/ \mathbb I_{<a}(r)$  is a consequence of 1. above
and $\dim H_r(X_a, X_{<a}) \geq\dim \mathbb T_{r-1}( <a,a) \geq  
\dim ^+\hat \lambda_{r-1}(a)$ follows from  (\ref{E9}).
\end{enumerate}
\end{obs}

\vskip.2in

{\bf Homological Tameness}
\vskip .1in

\begin{definition} 
\begin{enumerate}
\item
The continuous map $f$ is called {\it homologically tame} if for any $t\in \mathbb R$ there exists $\epsilon (t) >0$ s,t. the inclusion induced linear map $H_r(X_t)\to H_r(X_{t+\epsilon})$ and 
$H_r(X^t)\to H_r(X^{t-\epsilon})$ are isomorphisms for any $\epsilon, 0 < \epsilon <\epsilon(t).$  
\item
The continuous map $f$ is called {\it homologically finite } (better said left-homologically finite) 
if for any $t\in \mathbb R$  $\dim H_r(X_t, X_{<t}) <\infty.$
\end{enumerate}
\end{definition}

All real-valued maps on finite or  infinite dimensional manifolds considered in C.M.T are tame and homologically finite. 

By Observation \ref{O43},\ if $f$ is is homologically finite then   the maps $\delta^f_r ,$  $^+\gamma^f_r,$  $^+\mu_r,$  $^+\lambda^f_r$ and $^+\omega_r$  
 are $\mathbb Z_{\geq 0}-$valued functions with the  same support as $\hat \delta^f_r ,$  $^+\hat \gamma^f_r,$ $^+\hat \mu_r,$  $^+\hat \lambda_r$ and $^+\hat \omega^f_r.$  

\vskip .1in

The tameness  
implies  that the set $CR(f)$ is a discrete subset of $\mathbb R$ and therefore can  be totally ordered as a sequence $\{ \cdots  < c_i <c_{i+1} <\cdots\},$  possibly infinite in both directions,  and the filtration of $(C^{\delta, \gamma, \mu}_\ast , \partial ^{\delta, \gamma, \mu}_\ast)$ is locally constant in $t \in \mathbb R\setminus CR(f).$ If $f$ is in addition bounded from below then the sequence of critical values is bounded from below and each of the sub complexes $(C^{\delta, \gamma, \mu}_\ast (t), \partial ^{\delta, \gamma, \mu}_\ast)$
are chain complexes of finite dimensional  vector spaces, hence the considerations in Section 2 apply. 
 \vskip .1in 

 The following Theorem, whose proof is presented in the next section, establishes the relation between the maps $\hat \delta^f_r, ^+\hat \mu^f_r, ^+\hat \gamma^f_r,  ^+\hat \lambda^f_r$
 and the homology of the pair $(X_a, X_{<a}),$ and of the spaces   $X_a$ and $X.$ 

 \begin {theorem}\label {T44}\
 Suppose $f$ is  homologically finite and tame. Then
 \begin{enumerate}
\item  For $s \in \mathbb R\setminus CR(f),$  one has 
$H_r(X_s, X_{<s})=0$ 

and for $a\in CR(f)$ 

\begin{equation*}
\begin{aligned} 
H_r(X_a, X_{<a})= \bigoplus _{s\in CR(f)} \hat \delta^f_r(a,s) \oplus^+\hat \mu^f_r(a) &\oplus \bigoplus _{s\in (-\infty,a)\cap CR(f)}  {^+\hat \gamma}^f_{r-1}(s, a) \oplus ^+ \hat \lambda^f_{r-1} (a) 
\oplus \\
&\oplus  \bigoplus_{s\in (a, \infty) \cap CR(f) } {^+{\hat \gamma}^f_r(a,s) }
\end{aligned}
\end{equation*}

 \item If in addition $f$ is 
 bounded from below then  for any $t\in \mathbb R$ 
 \begin{enumerate} 
 \item $H_r(X_t)= \bigoplus _{(a,b) \mid a\leq t} \hat \delta^f_r(a,b)  
 \oplus \bigoplus _{(a,b)\mid a\leq t <b} {^+ \hat \gamma^f_r(a,b)} \oplus \bigoplus_{a\leq t}\ ^+\hat \mu ^f_r(a),$
 \item  $H_r(X)= \bigoplus _{(a,b)} \hat \delta^f_r(a,b) + \bigoplus _a \ ^+ \hat \mu_r(a) $ 
\end{enumerate}
with $a,b\in CR(f).$
 \item If in addition $f$ is  
 bounded from below and from above,
in particular if $X$ is compact,  then 
 \begin{enumerate} 
 \item $H_r(X_t)= \bigoplus _{(a,b) \mid a\leq t} \hat \delta^f_r(a,b)  
 \oplus \bigoplus _{(a,b)\mid a\leq t <b} {^+ \hat \gamma^f_r(a,b)} ,$
 \item  $H_r(X)= \bigoplus _{(a,b)} \hat \delta^f_r(a,b)$ 
 \end{enumerate}
 for $t\in \mathbb R$ and $a,b\in CR(f).$
\end{enumerate}
 \end{theorem} 

Note that  since $f$ is bounded from below, situation we meet in C.M.T, 
the sums in item 2(a)  contain only finitely many non vanishing terms.

The proof of this theorem is sketched in  Section 5.

\vskip .1in  

If $f:M\to \mathbb R$ is a proper smooth function on a finite dimensional manifold \footnote {or smooth function which satisfies Palais-Smale condition on an infinite dimensional manifold} with all critical points non degenerate  denote by  $c_r(f,a)$  the number of critical points of of index $r$ with $f(x)=a.$ Then one has:
\begin {proposition} (Morse Lemma) \label {P45}\
 $$c_r(f,a)= \dim H_r(M_a, M_{<a}).$$ 
\end{proposition}
Proposition (\ref{P45}) remains true if $M$ is an infinite dimensional smooth manifold and $f$ satisfies Palais-Smale condition  C \footnote{ any sequence of points $x_i\in M$ s.t. $d f (x_i)\to 0$ contains a subsequence convergent to  a critical point}. 
Theorem \ref{T44} and Proposition \ref{P45} imply the main result of this paper, Theorem \ref {TP}.
 
 \begin{theorem} \label {TP}\ 
 Suppose $f$ is a proper Morse function bounded from below  on a smooth manifold or  a smooth function on a smooth infinite dimensional manifold which satisfies Palais Smale condition whose critical points are non degenerate.
  Then for any $t$ the chain complexes $(C^{f, X,\mathcal O}_\ast  (t), \partial ^{X. \mathcal O}_\ast)$  and the chain complex $(C^{\delta, \gamma,\mu} (f)_\ast (t), \partial ^{\delta, \gamma, \mu}_\ast )$ are isomorphic. 
 \end{theorem}
 
 \begin{proof}
 Since both  complexes are complexes of finite dimensional vector spaces,  in view of Observation \ref{O21},  it suffices to check that the dimension of the $r-$components and the dimension of $r-$homology vector spaces are the same.
 
Indeed,
\begin{enumerate}[label= (\alph*)]
\item $\dim C^{f,X,\mathcal O}_r(t),$ which by Proposition \ref{P45} is equal to $\dim (\oplus_{a\leq t} H_r(X_a, X_{<a})),$ as implied by C.M.T.,
is  equal  to  $\dim (C_r^-(t) \oplus H_r(t) \oplus C_r(t)^+)= \dim C_r^{\delta, \gamma, \mu}(f) (t),$ cf. (\ref {E13}) as calculated by Theorem \ref {T44} item 1.
\item 
the dimension of the homology of the chain complex $(C_\ast^{f,X,\mathcal O}(t), \partial ^{X,\mathcal O}_\ast),$ which by  C.M.T is the dimension of $H_r(M_t),$  is  equal  to 
the dimension of the homology of the complex $(C_\ast^{\delta, \gamma, \mu}(t), \partial ^{\delta, \gamma, \mu}_\ast),$ as given by  Theorem \ref{T44} item 2.
\end{enumerate}
 \end{proof}

 \medskip
 One expects that the chain complexes $(C_\ast^{\delta, \gamma, \mu}, \partial ^{\delta, \gamma, \mu}_\ast)$ and $(C_r^{X,\mathcal O}, \partial ^{X,\mathcal O})$ are isomorphic by an isomorphism which preserves the $\mathbb R-$filtration;  this issue will be addressed later, being algebraically more involved.
 
 In particular, all conclusions about relationships between the rest points and instantons (for a vector field which admits a Morse function as Lyapunov) derived via homology,  can be derived from bar codes  and in  considerably more general situations, involving  more general spaces and more general flows.   \vskip ,2in

\section {About proofs} 

 Since in this section we will use direct and inverse limits it will be useful  to  have in mind the following facts. 
\begin {enumerate}
\item 
 (A1) For a direct system of short exact sequences of vector spaces the direct limit remains an exact sequence.
\item 
(A2) For an inverse system of short exact sequences $$\xymatrix{0\ar[r] &\{A_\alpha, ^A\iota_\alpha^{\alpha'} \} \ar[r] &\{B_\alpha, ^B\iota_\alpha^{\alpha'}\} \ar[r] &\{C_\alpha, ^C\iota_\alpha^{\alpha'}\} \ar[r]&0},$$  
by passing to the inverse limit one obtains  the exact sequence 
$$\xymatrix{ 0\ar[r]&\varprojlim_\alpha  A_\alpha \ar[r] &\varprojlim_\alpha  B_\alpha \ar[r] &\varprojlim_\alpha  C_\alpha \ar[r] &\varprojlim'_\alpha  A_\alpha \ar[r] &\cdots}$$
with  $\varprojlim'_\alpha  A_\alpha =0$ if the system $\{A_\alpha, ^A\iota_\alpha^{\alpha'} \}$ satisfies Mittag-Leffler condition, in particular if $^A\iota_\alpha^{\alpha'} $ are surjective or if $^A\iota_\alpha^{\alpha'} $ are injective and $A_\alpha$ are subspaces of a finite dimensional vector space $A.$
\end{enumerate}

\vskip .2in

For any $a, \alpha, \beta \in \mathbb R,  \alpha <\beta$ introduce 
\begin{equation}\label {E14}
\boxed{\begin{aligned}
\mathbb F^f_r(a\times [\alpha, \beta)):= 
 \mathbb F_r(a,\alpha ) / (\mathbb F_r(<a,\alpha) + \mathbb F_r(a, \beta)) .\\
 \end{aligned}} 
 \end{equation}
As shown in \cite{Bu2} one has the following Proposition.

\begin {proposition} \label {P51}\

 For any $a, b_1, b_2,  b_3 \in \mathbb R$ with $b_1 <b_2 <b_3$  one has the obviously induced linear maps $\iota$ and $\pi$ which provide  the short exact sequence 
$$\xymatrix{0 \ar[r] &\mathbb F_r (a\times[b_2, b_3) ) \ar[r]^i&\mathbb F_r (a\times [b_1, b _3)) \ar[r]^\pi &\mathbb F_r (a\times [ b_1, b_2) ) \ar[r]&0 . }$$
\end{proposition}
Unless obvious to the reader, details on the description of $\iota$ and $\pi$  and on the proof of this proposition  can be found in \cite {Bu2}.

In view of this Proposition (\ref{P51}) for any $x', x, y, y',  \ x' < x< y < y' $ one has the commutative diagrams 
$$\xymatrix{ \mathbb F_r(a\times [x',y)) & \mathbb F_r(a\times [x',y')) \ar@{->>}[l]\\\
 \mathbb F_r(a\times[x,y))\ar@{>->}[u]   & \mathbb F_r(a\times[x,y'))\ar@{>->}[u] \ar@{->>}[l], } 
,$$
both cartesian and co-cartesian with the horizontal arrows  surjective and the vertical arrows injective. 
This implies  (exercise left to the reader) the commutation of the direct and inverse limits stated as Observation \ref {O2}, cf  (\cite {Bu3}).
\begin {obs}\label {O2}\
The  diagram above  induces a canonical  isomorphism  
\begin{equation}  \label{E15}
\begin{aligned}
\varinjlim_{x\to -\infty} \varprojlim_{y\to \infty} \mathbb F_r(a \times [x,y))\to \varprojlim_{y\to \infty} \varinjlim_{x\to -\infty} \mathbb F_r(a\times [x,y)).\\
\end{aligned}
\end{equation}
\end{obs}

Similarly,  for any $a , b\in \mathbb R$ and $a \leq  \beta <\gamma$
and $ \alpha <\beta <b$ introduce  
\begin{equation} \label{E16}
\boxed{\begin{aligned}
\mathbb T_r(a\times (\beta,\gamma]):= &\mathbb T_r(a, \gamma)/ (\iota \mathbb T_r(<a,\gamma ) + \mathbb T_r(a, \beta))\\
 \mathbb T_r((\alpha,\beta]\times b):= &\mathbb T_r(\beta, b)/ (\iota \mathbb T_r(\alpha, b) + \mathbb T_r(\beta, <b))
\end{aligned}}
\end{equation}  
and one can verify as in \cite{Bu2}  the follwing, proposition: 
\begin{proposition} \label {P53}\

\begin{enumerate}
\item For any \ $a, b_1, b_2, b_3\in \mathbb R, \   a \leq b_1 <b_2<  b_3$
one has the obviously induced linear maps $\iota$ and $\pi$ which provide the short exact sequence 
$$\xymatrix{ 0\ar[r] &\mathbb T_r (a\times(b_1, b_2])\ar[r]^\iota & \mathbb T_r (a\times(b_1, b_3] )\ar[r]^{\pi}& \mathbb T_r (a\times (b_2, b_3]))\ar[r]&0 .}$$
\item For any $ a_1, a_2, a_3, b \in \mathbb R, \  a_1< a_2 <a_3 <  b$
one has the obviously induced linear maps $\iota$ and $\pi$ which provide the short exact sequence 
$$\xymatrix{ 0\ar[r] &\mathbb T_r ((a_1,a_2] \times b)\ar[r]^\iota & \mathbb T_r ((a_1,a_3]\times b )\ar[r]^\pi& \mathbb T_r ((a_2, a_3] \times b))\ar[r]& 0 .}$$ 
\end{enumerate}
\end{proposition}
Details on the description of $\iota$ and $\pi$ and the verifications of this proposition can be found in \cite{Bu2} or \cite {Bu}.

In view of this Proposition (\ref{P53}), for $a< y' < y< x ' <x$  resp. $x' <x <y<y' <b,$ one has the commutative diagrams 
$$\xymatrix{ \mathbb T_r(a\times (y,x]) & \mathbb T_r(a\times (y',x]) \ar@{->>}[l]\\\
 \mathbb T_r(a\times(y,x'])\ar@{>->}[u]   & \mathbb T_r(a\times(y', x'])\ar@{>->}[u] \ar@{->>}[l] } \ \   \xymatrix{ \mathbb T_r((x,y']\times b) & \mathbb T_r((x',y']\times b) \ar@{->>}[l]\\\
 \mathbb T_r((x,y]\times b)\ar@{>->}[u]   & \mathbb T_r((x',y]\times b)\ar@{>->}[u] \ar@{->>}[l] ,}      $$
both cartesian and co-cartesian, whose horizontal arrows are surjections and the vertical arrows are injections. 
As before they imply the commutation of the direct and inverse limit stated below, cf \cite {Bu3}.

\begin{obs} \label {O4}\ 
The diagrams above induce the canonical  isomorphisms  
\begin{equation}\label{E17}
\varprojlim_{x\to a} \varinjlim_{y\to \infty} \mathbb T_r(a \times (x,y])\to \varinjlim_{y\to \infty} \varprojlim_{x\to a} \mathbb T_r(a\times (x,y]), \\
\end{equation}
\begin{equation}\label {E18}
\varinjlim_{y\to b} \varprojlim_{x\to -\infty} \mathbb T_r((x,y]\times b)\to \varprojlim_{x\to -\infty} \varinjlim_{y\to b} \mathbb T_r((x,y]\times b).
\end{equation}
\end{obs}

If one defines $\mathbb F_r(a\times (-\infty, \infty))$ as either one of the two limits (\ref {E15}), in view of (\ref{E14}), one concludes that 
 \begin{equation*} 
\begin{aligned}
\mathbb F_r(a \times (-\infty,\infty)):=&\varprojlim_{y\to \infty} \varinjlim_{x\to -\infty} \mathbb F_r(a \times [x,y)) =
\varprojlim_{y\to \infty}\mathbb I_a(r) / ( \mathbb I_{<a}(r)  + \mathbb I_a(r)\cap  \mathbb I^{y}(r)).
\end{aligned}
\end{equation*}  

By  passing to inverse limit  ($y\to \infty$), in view of A2 and of the finite dimensionality of $\mathbb I_a(r) /  \mathbb I_{<a}(r),$
the exact sequence
$$\xymatrix{0\ar[r]&(\mathbb I_a(r)\cap \mathbb I^{y}(r)) / (\mathbb I_{<a}(r)\cap \mathbb I^{y}(r))\ar[r]  &\mathbb I_a(r) /  \mathbb I_{<a}(r)\ar[r]&\mathbb I_a(r) / ( \mathbb I_{<a}(r)+ \mathbb I_a(r)\cap \mathbb I^{y}(r)) \ar[r]&0}$$
remains exact. One obtains 
$$ 0 \to
\varprojlim_{y\to \infty} (\mathbb I_a(r)\cap \mathbb I^{y}(r)) / (\mathbb I_{<a}(r)\cap \mathbb I^{y}(r))\to  \mathbb I_a(r) /  \mathbb I_{<a}(r)\to  
\varprojlim_{y\to \infty}\mathbb I_a(r) / ( \mathbb I_{<a}(r)+ \mathbb I_a(r)\cap \mathbb I^{y}(r)) 
\to 0.$$
Since  
$$ \varprojlim_{y\to \infty} (\mathbb I_a(r)\cap \mathbb I^{y}(r)/ (\mathbb I_a(r)\cap \mathbb I^{y}(r))\simeq^+\hat \mu_r(a)$$
and 
$$\varprojlim_{y\to \infty}\mathbb I_a(r) / ( \mathbb I_{<a}(r)+ \mathbb I_a(r)\cap \mathbb I^{y}(r))= \mathbb F_r(a\times (-\infty, \infty)$$
one obtains 
\begin{equation}\label{E19}
\mathbb I_a(r) /  \mathbb I_{<a}(r) \simeq  \mathbb F_r(a \times (-\infty,\infty)) \oplus   ^+\hat \mu_r(a).
\end{equation}    
\vskip .1in
 Similarly, 
if one defines $\mathbb T_r ((-\infty, a)\times a)$
 using either one of the limits in (\ref{E18}), for example 
 $$\mathbb T_r ((-\infty, a)\times a)=  \varinjlim_{y\to a} \varprojlim_{x\to -\infty}\mathbb T_r ((x,y]\times a),$$
when $f$ is homologically finite,
one obtains 
 \begin{equation}\label {E20} 
\mathbb T_r(<a,a)= \mathbb T_r((-\infty,a)\times a)  \oplus ^+\lambda^f_r(a).
\end{equation}
Indeed,  by  passing to the inverse inverse limit $(x\to -\infty)$ in the short exact sequence 
$$\xymatrix {0\ar[r] &^a\iota _x^y\mathbb T_r(x,a)/ ^{<a}\iota^y_x \mathbb T_r(x, <a) \ar[r] &\mathbb T_r(y,a)/ \mathbb T_r(y, <a)\ar[r] &\mathbb T_r(y,a)/ (\mathbb T_r(y, <a) +^a\iota_x^y \mathbb T_r(x, a))\ar[r] &0},$$  
since  $\dim \mathbb T_r(y,a)/ \mathbb T_r(y, <a)$ is bounded \footnote { by Observation \ref {O43} item 3    $\dim \mathbb T_r(y,a)/ \mathbb T_r(y, <a)$  is finite and bounded  by $\dim H_r(X_a, X_<a)$}, 
in view of A2 one obtains 
$$\xymatrix {0\ar[r] &^+\lambda^f_r(a)\ar[r] &\mathbb T_r(y,a)/ \mathbb T_r(y, <a)\ar[r] &\varprojlim_{x\to -\infty}  \mathbb T_r(y,a)/ (\mathbb T_r(y, <a) +\iota \mathbb T_r(x, a))\ar[r] &0}$$ 
 and the by passing to the direct limit $(y\to a)$ in view of A1
 one obtains  (\ref{E20}).
\vskip .1in

Similarly, 
 one defines $\mathbb T_r (a\times (a,\infty))$
 using either one of the limits in (\ref{E17}), for example 
 $$\mathbb T_r (a\times (x,y])=  \varinjlim_{y\to \infty} \varprojlim_{x\to a}\mathbb T_r (a\times (x,y]).$$ 
 By passing to the inverse limit $(x\to a )$ in the short exact sequence
$$\xymatrix {0\ar[r] & \mathbb T_r(a,x)/ \iota \mathbb T_r(<a, x ) \ar[r] &\mathbb T_r(a,y)/ \iota \mathbb T_r( <a,y)\ar[r] &\mathbb T_r(a,y)/ (\iota \mathbb T_r( <a,y) +\mathbb T_r(a,x))\ar[r] &0},$$  
because 
$\dim \mathbb T_r(a,x)/ \mathbb T_r( <a,x)< 
\dim H_r(X_a, X_{<a}) <\infty,$ one obtains in view of A2
$$\xymatrix {0\ar[r] &^+\omega^f_r(a)\ar[r] & \mathbb T_r(a,y)/ \mathbb T_r(( <a,y) \ar[r] &\varprojlim_{x\to a} \mathbb T_r(a,y)/ (\mathbb T_r(<a,y) +\iota \mathbb T_r(a,x))\ar[r] &0}.$$  
Then, by  passing to the direct limit $(y \to \infty),$  in view of A1,one obtains
$$
\xymatrix {0\ar[r] &^+\omega^f_r(a)\ar[r] & \mathbb T_r(a,\infty)/\iota  \mathbb T_r(( <a,\infty)\ar[r] &\mathbb T_r(a\times (a,\infty))\ar[r] &0,} 
$$
hence 
\begin{equation} \label {E21}
\mathbb T_r(a,\infty)/\iota  \mathbb T_r(( <a,\infty)\simeq ^+\omega^f_r(a)\oplus \mathbb T_r(a\times (a,\infty)).
\end{equation}

 The commutative diagram $$\xymatrix{
          &0                                &0                            &0 &\\
0\ar[r]&\mathbb I_{<a}(r)\ar[r]\ar[u]&\mathbb I_a(r) \ar[r]\ar[u]&\mathbb I_a(r) /\mathbb I_{<a}(r)\ar[r]\ar[u] &0\\
          & H_r(X_{<a}) \ar[r]\ar[u]&H_r(X_a)\ar[r] \ar[u]& \coker (H_r(X_{<a})\to H_r(X_a))\ar[r]\ar[u]&0\\
          &\mathbb T_r(<a,\infty)\ar[r]\ar[u]&\mathbb T_r(a, \infty)\ar[r] \ar[u]&\mathbb T_r(a, \infty)/  \iota\mathbb T_r(<a, \infty)\ar[r]\ar[u]&0\\
          &0\ar[u]                                       &0\ar[u]                                      &0\ar[u]&}. $$                                                                                                                        

whose  all raws and columns are exact sequences
implies the isomorphism 

\begin{equation} \label{E22}
 \coker (H_r(X_{<a})\to H_r(X_a))\simeq \mathbb I_a(r) /\mathbb I_{<a}(r)\oplus \mathbb T_r(a, \infty)/  \iota\mathbb T_r(<a, \infty).
\end{equation}

The exact homology sequence of the pair $(X_\alpha, X_{<a})$ induces the isomorphism
 \begin{equation} \label {E23}
H_r(X_a, X_{<a})= \coker (H_r(X_{<a}) \to H_r (X_a)) \oplus T_{(r-1)}( <a, a).
\end{equation}

Consequently for $f$ homologically finite  (\ref{E23}),  (\ref{E22}) followed by (\ref {E19}),  and (\ref{E20})  imply 
\begin{equation}\label {E24}
H_r(X_a, X_{<a})\simeq \mathbb F_r(a \times (-\infty,\infty)) \oplus  \mathbb T_r(a\times (a,\infty))
\oplus \mathbb T_{r-1}((-\infty,a)\times a)\oplus  ^+\hat \mu_r(a)\oplus ^+\omega^f_r(a)\oplus  ^+\lambda^f_{r-1}(a)   \end{equation}

\begin {proposition} \label {P55}\
Suppose $\dim H_r( X_a, X_{<a}) <\infty.$ Then: 
\begin{enumerate} 
\item  
There exists a finite set of real numbers $\{x_1, x_2, \cdots x_r\}$ depending on $a$ s.t.with $-\infty <x_1 <x_2 \cdots <x_r  < \infty$  s.t. 
\begin{enumerate} 
\item $\hat \delta^f_r(a,x)=0$ of $x\ne \{ x_1, \cdots x_r\},$
and if in addition $f$ is tame then  
 \item $\mathbb F_r(a\times [x, y))\simeq \bigoplus_{\{i\mid x\leq x_i <y\}}\hat \delta^f_r(a, x_i).$
\end{enumerate}
\item 
There exists a finite set of real numbers $\{y_1, y_2, \cdots y_r\}$ depending on $a$ with $a <y_1 <y_2 \cdots y_r < \infty$ s.t. 
\begin{enumerate} 
\item $^+\hat \gamma ^f_r(a,y)=0$ of $y\ne \{ y_1, \cdots y_r\}$ 
and if in addition $f$ is tame then 
\item $\mathbb T_r(a\times (x, y])\simeq \bigoplus_{\{i\mid x< y_i \leq y\}}  ^+\hat \gamma^f_r(a, y_i).$
\end{enumerate}
\item 
There exists a finite set of real numbers $\{x_1, x_2, \cdots x_k\}$ depending on $a$ with $-\infty <x_1 <x_2 \cdots x_k< a$ s.t. 
\begin{enumerate} 
\item $\hat \delta^f_{r-1}(x, a)=0$ of $x\ne \{x_1, \cdots x_k\}$  
and if in addition $f$ is tame then 
\item $\mathbb T_r((x,y]\times a)\simeq \bigoplus_{\{i\mid -\infty <x_i \leq y\}} ^+\hat \gamma^f_r(x_i,a). $
\end{enumerate}
\end{enumerate}
\end{proposition}

The proof is presented in \cite {Bu2}  but  the for the reader's convenience is also sketched below.

\vskip ,2in

Note that in view of Proposition \ref{P51} the $\mathbb Z-$valued function  
$\dim \mathbb F_r(a\times [x,y)), \ x<y$ increases when $y$ increases to $\infty $ 
or  $x$ decreases  to $-\infty$ and the other variable remains  constant.  

Similarly in view of Proposition \ref{P53} the function $\dim \mathbb T_r(a\times (x,y]),  \ a<x<y$
increases,  when $y$ increases  to $\infty$ 
or when or 
$x$ decreases to  $a$ 
and the other variable remains constant
and   
the function $\dim \mathbb T_r((x,y]\times b), \ x<y<b$ increases  when 
$y$ increases to $b$  
or  when $x$ decreases to $-\infty$ and the other variable remains constant.

Note also that  if $f$ is tame then the inclusion  $\mathbb I_f^y \subseteq \mathbb I_f^{y-\epsilon}$ is an isomorphism for $\epsilon$ small enough and  in view of Proposition \ref{P51},  the function $\dim \mathbb F_r(a\times [x,y))$ is continuous from the left in $x$ and in $y$  when $y$ resp. $x$  remains constant.

Similarly if $f$ is tame then the inclusion  $\mathbb T_r(x,b) \subseteq  T_r(x+\epsilon, b)$ 
is an  isomorphism for $\epsilon$ small enough
and  in view of Proposition \ref{P53} the function $\dim \mathbb T_r((x,y] \times b)$ is continuous from the right in $y$  and in $x$  when $x$ resp. $y$ remains constant. 

Also  if $f$ is tame then the inclusion  $\mathbb T_r(a, y)\subseteq  T_r(a, y+\epsilon)$  
 is an isomorphism for $\epsilon$ small enough
 in view of Proposition \ref{P53} the function $\dim \mathbb T_r(a\times (x,y])$ is continuous from the right in both variables $x$ and $y$  when $y$ resp. $x$ remains constant.
\vskip .1in

Then 
the jumps at $y$  for the first two functions  and  at $x$ for the third function, when appear, in view of Propositions \ref{P51} and \ref{P53} are given by 
\begin{equation}\label{E25}
\begin{aligned}  
\lim _{\epsilon \to 0}  (\dim \mathbb F_r(a\times [x, y+\epsilon))-  \dim \mathbb F_r(a\times [x, y)))=& \lim _{\epsilon \to 0}  \dim \mathbb F_r(a\times [y, y+\epsilon)= \delta^f_r(a, y),\\
\lim _{\epsilon \to 0}  (\dim \mathbb T_r(a\times (x, y)-  \dim \mathbb T_r(a\times (x, y-\epsilon]))= &\lim _{\epsilon \to 0}  \dim \mathbb T_r(a\times (y-\epsilon, y])=^+ \gamma^f_r(a, y),\\
\lim _{\epsilon \to 0}  (\dim \mathbb T_r ((x-\epsilon, y] \times a )-  \dim \mathbb T_r((x, y] \times a))= &\lim _{\epsilon \to 0}  \dim \mathbb T_r((x-\epsilon, x] \times a )= ^+\gamma^f_r(x,a).
\end{aligned}
\end{equation}
\vskip .1in

In view of Propositions (\ref{P51}) and (\ref{P53}) and of tameness of $f$ observe that:
\begin{enumerate}
\item $dim \mathbb F_r(a\times [x,y))$ = 0 for $x$ very large towards $\infty$, or for $y<-K$ with  $K$ large enough  towards $\infty ,$ 
 
\item $dim \mathbb T_r(a\times (x,y])$ = 0 for $x$ large enough towards $\infty$ or for $x$ closed enough  to $a$   
 
\item $dim \mathbb T_r((x,y]\times a)$ = 0 for $y$ closed enough to $a$  or for $x <-K$ with $K$ large enough  towards $\infty.$  
\end{enumerate}
 and  then

\begin{equation}\label {E26}
\mathbb \mathbb F_r(a\times (-\infty, \infty))= \bigoplus_{t\in \mathbb R } \hat \delta_r(a, t) 
\end{equation}

\begin{equation}\label {E27}
 \mathbb T_r (a\times (a, \beta] )= \bigoplus _{t\mid a < t \leq \beta}  \ ^+\hat \gamma^f_r(a,t),   a <\beta \in \mathbb R
\end{equation}

\begin{equation}\label {E28}
\mathbb T_r(a\times (\beta,\infty))= \varinjlim _{y\to \infty} \mathbb T_r(a\times (\beta, y])=\oplus_{ s>\beta} \ ^+ \hat \gamma^f_r(a, s)
\end{equation} 

\begin{equation}\label {E29}
\mathbb \mathbb T_{r-1} ((-\infty,a)\times a)= \bigoplus_{t\in (-\infty, a)} 
\hat \gamma_{r-1} (t, a) 
\end{equation}

{\bf Proof of Theorem (\ref {T44}) }
\vskip .1in 
\begin{proof} 
First recall that : 
\begin{enumerate} 
\item If $f$ is tame then 
$^+\hat \omega^f_r(a)=0,$ 
\item If $f$ is bounded from below then $^+\hat \lambda^f_r(a)=0,$
\item If $f$ is bounded from above  then $^+\hat \mu^f_r(a)=0.$
\end{enumerate}

Then Item 1. in Theorem \ref{T44} follows on the nose from (\ref {E24}), (\ref {E26}),(\ref {E28}), (\ref {E29}) and (1.) above.

Note that since   $f$ is tame and bounded from below,  for any $a$ there are only finitely many critical values smaller than $a,$  say $c_1, c_2, \cdots c_k\leq a,$  hence a finite numbers of values $s,$ $s\leq a,$ such that 
 the vector space $\mathbb I_s/ \mathbb I_{<s}$ is not trivial and all these vector spaces are of finite dimension. Then $\mathbb I_a(r) = \oplus _{s\leq a}
\mathbb I_s/ \mathbb I_{<s}$ and  then, in view of (\ref{E19}) and (\ref{E26}) 

\begin{equation}\label {E30}
\mathbb I_a(r)= \oplus_{s\leq a}\mathbb I_s(r) / \mathbb I_{<s}(r)= \bigoplus _{s\leq a, t\in \mathbb R} \hat \delta^f_r(s,t) \oplus \bigoplus_{s \leq a}\  ^+\hat \mu^f_r(s)
.\end{equation}
To check  Item 2. Theorem \ref{T44}  note  that $H_r(X_a)= \mathbb I_a(r)\oplus \mathbb T_r(a,\infty)$ 
and in order to calculate  $\mathbb T_r(a,\infty) $ observe the following: 
\begin{enumerate} [label= (\roman*)]

\item Since $f$ is tame  in view of (\ref{E28}) one has 
\begin{equation*}
\mathbb T_r(a\times (\beta,\infty))= \varinjlim _{y\to \infty} \mathbb T_r(a\times (\beta, y])=\oplus_{ s>\beta} \ ^+ \hat \gamma^f_r(a, s).
\end{equation*} 
\item Since $f$ is  also  bounded from below for any $a$ the set of critical values smaller or equal to $a$ is finite.  Then let $c_1 <c_2 <\cdots c_k \leq a$ be these critical values and 
observe  that 
in view of Proposition \ref {P55} item 2  the set 
$\bigcup _{1\leq i \leq k} \supp \ ^+\hat \gamma_r \cap (c_i \times (c_i, \infty)) \subset \mathbb R^2_+$ 
is finite.
Then  $\mathbb T_r(a,\infty)= \bigoplus _{1\leq i\leq k} \mathbb T_r (c_i\times (c_i,\infty)) \footnote {because $\mathbb T_r (a\times (a,\infty))= \mathbb T_r(a,\infty)/ \mathbb T_r(<a,\infty)$} $  and 
 in view of (\ref{E28}) we have  
 \begin{equation}\label {E31}
\mathbb T_r(a,\infty)= \bigoplus _{1\leq i\leq k} \mathbb T_r (c_i\times (c_i,\infty)) = \bigoplus _{1\leq i\leq k} ( \bigoplus_ {c_i<s} \ ^+ \hat \gamma^f_r(c_i, s) = \bigoplus _{(t,s) \mid t\leq a} \ ^+\hat \gamma^f_r(t, s).
\end{equation} 
\end{enumerate}

Combining (\ref{E30}) and (\ref{E31}) one completes  the proof of part (a) in Item 2. of Theorem \ref{T44}. To prove  part (b) one pass to the direct limit when $a\to \infty.$replaces $X$ by $X_b$ and $\infty$  by $b$ and one repeats the arguments.

Item 3 of Theorem \ref{T44} follows from Item 2  and 2. and 3. above. 

\end{proof}
\vskip .1in

{\bf Proof of Proposition \ref{P45}}

For each critical point $x_1, x_2, \cdots x_p$ located on the level $f^{-1} (a)$ choose a Morse chart  for $f$ 
and choose $\epsilon$ small enough such that the discs of radius $\epsilon$  (in these charts) provide closed disjoint neighborhoods of the critical points.  Denote each such neighborhood by $D(i)$  with  $x_i\in  D(i)$ and let $D=\sqcup D(i).$  The set of critical points $\{x_1, x_2, \cdots x_p\}$ is contained in the interior of $D.$ Let $r$ be the number of critical points of index $r;$  clearly $r\leq p.$  Let $M':= M\setminus \{x_1, \cdots x_p\}$ and observe that $M'_a$ is a manifold with boundary whose interior is $M_{<a},$ hence  $H_r( M'_{a}, M_{<a})=0$ for any $r.$ Then $H_r (M_a, M_{<a})= H_r( M_a, M'_{a})$ which by excision is isomorphic to $H_r( D_a, D_a\setminus \{x_{i_1}, \cdots x_{i_r}\})$, $x_{i_1}. \cdots x_{i_r}$ the critical points of index $r,$  $D_a= f^{-1} (-\infty, a])\cap D$, which by Morse lemma is isomorphic to the direct sum of $r$ copies of $H_r(D^r, \partial D^r)=\kappa.$ 

\end{document}